\documentclass[12pt]{amsart}

\makeatletter
\def\@settitle{\begin{center}\baselineskip14\p@\relax\sc\Large\@title\end{center}}
\def\@setauthors{\begingroup\def\thanks{\protect\thanks@warning}\trivlist\centering\footnotesize \@topsep30\p@\relax\advance\@topsep by -\baselineskip\item\relax\author@andify\authors\def\\{\protect\linebreak}\sc\large{\authors}\ifx\@empty\contribs\else,\penalty-3 \space \@setcontribs\@closetoccontribs\fi\endtrivlist\endgroup}
\renewcommand*{\@seccntformat}[1]{\csname the#1\endcsname.\hspace{1mm}} 
\makeatother
\usepackage[top=1.25in,left=1.6in,right=1.6in,bottom=1in,paper=a4paper]{geometry}

\title[Embeddability of path geometries]{Local Embeddability of Real Analytic Path Geometries}
\author[T. Mettler]{Thomas Mettler}
\address{Mathematical Sciences Research Institute, 17 Gauss Way, Berkeley, CA-94720, USA}
\email{tmettler@msri.org}
\date{March 16, 2012}

\thanks{Research for this article was carried out while the author was supported by Schweizerischer Nationalfonds SNF via the postdoctoral fellowship PBFRP2-133545 and by the Mathematical Sciences Research Institute, Berkeley.}
\keywords{Path geometries, CR-structures, embedding problems, exterior differential systems.}
\subjclass[2010]{58A15, 32V30, 53C10}

\renewcommand{\rm}{\mathrm}
\renewcommand{\Re}{\operatorname{Re}}
\renewcommand{\Im}{\operatorname{Im}}
\newcommand{\End}{\operatorname{End}}
\newcommand{\J}{J}
\newcommand{\AJ}{\mathfrak{J}}
\newcommand{\M}{M}
\newcommand{\N}{N}
\newcommand{\B}{\Lambda_\AJ}
\newcommand{\ba}{b}

\renewcommand{\O}{H}

\newcommand{\R}{\mathbb{R}}

\newcommand{\bT}{T}
\newcommand{\C}{\mathbb{C}}
\newcommand{\bv}{v}

\newcommand{\Ad}{\operatorname{\textrm{Ad}}}

\newcommand{\bX}{X}
\renewcommand{\d}{d}

\renewcommand{\d}{\mathrm{d}}
\renewcommand{\i}{\mathrm{i}}

\newtheorem{theorem}{Theorem}
\newtheorem{corollary}{Corollary}
\newtheorem{lemma}{Lemma}
\newtheorem{proposition}{Proposition}

\theoremstyle{definition}
\newtheorem*{defi}{Definition}

\newtheorem*{remark}{Remark}

\numberwithin{equation}{section}

\begin{document}

\begin{abstract}
An almost complex structure $\AJ$ on a $4$-manifold $X$ may be described in terms of a rank $2$ vector bundle $\Lambda_\AJ \subset \Lambda^2TX^*$. We call a pair of line subbundles $L_1,L_2$ of $\Lambda^2TX^*$ a splitting of $\AJ$ if $\Lambda_{\AJ}=L_1\oplus L_2$. A hypersurface $M \subset X$ satisfying a nondegeneracy condition inherits a CR-structure from $\AJ$ and a path geometry from the splitting $(L_1,L_2)$. Using the Cartan-K\"ahler theorem we show that locally every real analytic path geometry is induced by an embedding into $\C^2$ equipped with the splitting generated by the real and imaginary part of $\d z^1\wedge \d z^2$. As a corollary we obtain the well-known fact that every $3$-dimensional nondegenerate real analytic CR-structure is locally induced by an embedding into $\C^2$.  
\end{abstract}

\maketitle

\section{Introduction}
Motivated by the well-known fact (see for instance~\cite{MR2313334}) that an almost complex structure $\AJ$ on a $4$-manifold $X$ admits a description in terms of a rank $2$ vector bundle $\Lambda_\AJ \subset \Lambda^2TX^*$, we introduce the notion of a splitting of an almost complex structure:  A pair of line subbundles $L_1,L_2$ of $\Lambda^2TX^*$ is called a \textit{splitting} of $\AJ$ if $\Lambda_{\AJ}=L_1\oplus L_2$. A hypersurface $M \subset X$ satisfying a nondegeneracy condition inherits a CR-structure from $\AJ$ and a path geometry from the splitting $(L_1,L_2)$. The purpose of this Note is to show that locally every real analytic path geometry is induced by an embedding into $\R^4\simeq\C^2$ equipped with the splitting generated by the real and imaginary part of $\d z^1\wedge \d z^2$. This will be done using the Cartan-K\"ahler theorem. As a corollary we obtain the well-known fact that every $3$-dimensional nondegenerate real analytic CR-structure is locally induced by an embedding into $\C^2$. It follows with Nirenberg's example of a smooth non-embeddable $3$-dimensional CR-manifold that the real analyticity in our main statement is necessary.  

The notation and terminology for the Cartan-K\"ahler theorem and exterior differential systems are chosen to be consistent with~\cite{MR1083148,MR2003610}. Moreover we adhere to the convention of summing over repeated indices. 

\section{Preliminaries}

\subsection{Pairs of 2-forms}

Throughout this section, let $V$ denote an oriented $4$-dimensional real vector space. Fix a volume form $\varepsilon \in \Lambda^{4}V^*$ which induces the given orientation. Given two $2$-forms $\omega,\phi \in \Lambda^{2}V^*$, we may write
$
\omega\wedge\phi=\langle \omega,\phi \rangle \varepsilon
$
for some unique real number $\langle \omega,\phi \rangle$. Clearly the map $(\omega,\phi) \mapsto \langle \omega,\phi\rangle$ defines a symmetric bilinear form on the $6$-dimensional real vector space $\Lambda^{2}V^*$ which is easily seen to be nondegenerate and of signature $(3,3)$. Replacing $\varepsilon$ with another orientation compatible volume form gives a bilinear form which is a positive multiple of $\langle\cdot\,,\cdot\rangle$. Consequently, the wedge product may be thought of as a conformal structure of split signature on $\Lambda^2V^*$. 
\begin{defi}
A pair of $2$-forms $\omega,\phi \in \Lambda^2V^*$ is called \textit{elliptic} if 
$$
\langle \omega,\omega \rangle\langle \phi,\phi \rangle>\langle \omega,\phi \rangle^2.
$$
\end{defi}
It is a natural problem to classify the pairs of elliptic $2$-forms on $V$. This is a special case of a more general problem: Let $\omega \in \Lambda^2V^*$ be a symplectic $2$-form whose stabiliser subgroup will be denoted by $\rm{Sp}(\omega) \subset \rm{GL}(V)$. The natural representation of $\rm{Sp}(\omega)$ on $\Lambda^2V^*$ decomposes as
$
\Lambda^2V^*=\left\{\omega\right\} \oplus \omega^{\perp}
$
where both summands are irreducible $\rm{Sp}(\omega)$-modules.\footnote{We denote by $\left\{\cdot \right\}$ the linear span of the elements within.  In the case of smooth differential forms, the coefficients are smooth real-valued functions.} Here $\omega^{\perp}$ is the $5$-dimensional linear subspace of $\Lambda^2V^*$ consisting of $2$-forms orthogonal to $\omega$. One can ask to classify the orbits of $\rm{Sp}(\omega)$ on $\omega^{\perp}$. This has been carried out in~\cite{MR703049}. In the elliptic case one obtains:

\begin{lemma}[see~\cite{MR703049}]\label{alg}
Let $\omega,\phi \in \Lambda^2V^*$ be a pair of elliptic orthogonal $2$-forms, then there exists a positive real number $\kappa$ and a basis $e^i$ of $V^*$ such that 
$$
\omega=e^1\wedge e^3-e^2 \wedge e^4, \quad \phi=\kappa\left(e^1\wedge e^4+e^2\wedge e^3\right). 
$$
\end{lemma}
The constant $\kappa$ is an $\rm{Sp}(\omega)$-invariant and thus parametrises the set of elliptic $\rm{Sp}(\omega)$-orbits. Ellipticity will be useful because of the following: 
\begin{lemma}\label{huereseich}
Let $W$ be $3$-dimensional real vector space. Then the pullback of an elliptic pair of $2$-forms $\omega,\phi \in \Lambda^2V^*$ with any injective linear map $A : W \to V$ gives two linearly independent $2$-forms on $W$. 
\end{lemma}
\begin{proof}
The ellipticity condition is equivalent to every nonzero linear combination of $(\omega,\phi)$ being symplectic. Suppose $(\omega,\phi)$ is an elliptic pair of $2$-forms. Then for every choice of real numbers $(\lambda_1,\lambda_2)\neq 0$, the $2$-form $\tau=\lambda_1\omega+\lambda_2\phi$ is symplectic. Since there are no isotropic subspaces of dimension greater than $2$ in the symplectic vector space $(V,\tau)$, it follows that $A^*\tau=\lambda_1A^*\omega+\lambda_2A^*\phi\neq 0$ for every linear injective map $A : W \to V$.   
\end{proof}

\subsection{Splittings of complex structures} 

Let  $\mathcal{C}^+(V)$ denote space of complex structures on $V$ which are compatible with the orientation, i.e. its points $\J \in \End(V)$ satisfy $\varepsilon(v_1,\J v_1,v_2,\J v_2)\geq 0$ for all vectors $v_1,v_2 \in V$. Moreover let $G^+_2(\Lambda^2V^*,\wedge_+)$ denote the submanifold of the Grassmannian of oriented $2$-planes in $\Lambda^2V^*$ to whose elements the wedge product restricts to be positive definite. Given a $(2,\! 0)$-form $\alpha \in \Lambda^{2,0}V^*$ with respect to some $J \in \mathcal{C}^+(V)$, let $\Lambda_J \in G^+_2(\Lambda^2V^*,\wedge_+)$ denote the $2$-dimensional linear subspace spanned by $\Re(\alpha)$, $\Im(\alpha)$ and orient $\Lambda_J$ by declaring $\Re(\alpha),\Im(\alpha)$ to be positively oriented. Clearly $\Lambda_J$ and its orientation are independent of the chosen $(2,\! 0)$-form $\alpha$ and one thus obtains a map $\psi : \mathcal{C}^+(V) \to G^+_2(\Lambda^2V^*,\wedge_+)$ given by $J \mapsto \Lambda_{J}$. Note that $G=\rm{GL}^+(V)$ acts smoothly and transitively from the left on $\mathcal{C}^+(V)$ via $(A,J) \mapsto A^{-1}JA$. Every element of $G^+_2(\Lambda^2V^*,\wedge_+)$ admits a positively oriented elliptic conformal basis. It follows with Lemma \ref{alg} that via pushforward, $\rm{GL}^+(V)$ acts smoothly and transitively from the left on $G^+_2(\Lambda^2V^*,\wedge_+)$ as well.   
\begin{proposition}\label{equi}
The map $
\psi : \mathcal{C}^+(V) \to G^+_2(\Lambda^2V^*,\wedge_+)$, $J \mapsto \Lambda_J$ is a $G$-equivariant diffeomorphism. 
\end{proposition}
\begin{proof}
Clearly the map $\psi$ is $G$-equivariant. To prove that $\psi$ is a diffeomorphism it is sufficient to show that $G_J=G_{\psi(J)}$ for all $J \in \mathcal{C}^+(V)$ where $G_J$ and $G_{\psi(J)}$ denote the stabiliser subgroups of $G$ with respect to $J$ and $\psi(J)$ respectively. Choose $J \in \mathcal{C}^+(V)$, then we have $G_J \subset G_{\psi(J)}$. Write 
$$
J(v)=-e^2(v)e_1+e^1(v)e_2-e^4(v)e_3+e^3(v)e_4
$$
for some basis $(e_i)$ of $V$ and dual basis $(e^i)$ of $V^*$. Then 
$$
\omega=e^1\wedge e^3-e^2\wedge e^4=\frac{1}{2}w_{kl}e^k\wedge e^l, \quad \phi=e^1\wedge e^4+e^2\wedge e^3=\frac{1}{2}f_{kl}e^k\wedge e^l
$$
is a positively oriented conformal basis of $\Lambda_J$. Consequently every $A \in G_{\psi(J)}$ satisfies
$
A^*\omega=x\omega+y\phi$ and $A^*\phi=-y\omega+x\phi,
$
for some real numbers $(x,y)\neq 0$. The matrix representation $a$ of $A$ with respect to the basis $(e_i)$ thus satisfies
$$
a^twa=xw+yf, \quad a^tfa=-yw+xf. 
$$
From this one easily concludes $awf=wfa$ which is equivalent to $A$ commuting with $J$.  
\end{proof}
Proposition \ref{equi} motivates the following: 
\begin{defi}
A \textit{splitting} of a complex structure $J$ on $V$ is a pair of lines $L_1,L_2 \in \mathbb{P}(\Lambda^2V^*)$ such that $\Lambda_{J}=L_1\oplus L_2$. 
\end{defi}
Call two $4$-dimensional real vector spaces $V$, $V^{\prime}$ equipped with complex structures $J$, $J^{\prime}$ and splittings $(L_1,L_2)$, $(L_1^{\prime},L_2^{\prime})$ equivalent, if there exists a complex linear map $A : V \to V^{\prime}$ such that $A^*(L_i^{\prime})=L_i$ for $i=1,2$. 

On $V=\R^4$ let $\omega_0=e^1\wedge e^3-e^2\wedge e^4$ and $\phi_0=e^1\wedge e^4+e^2\wedge e^3$ where $e^1,\ldots,e^4$ denotes the standard basis of $(\R^4)^*$. Define $L_1=\left\{\omega_0\right\}$ and $L_2=\left\{\alpha \omega_0+\phi_0\right\}$ for some nonnegative real number $\alpha$. Orient $L_1\oplus L_2$ by declaring $\omega_0,\phi_0$ to be a positively oriented basis and let $J_0$ be the associated complex structure. Then $S_\alpha=(L_1,L_2)$  is a splitting of $J_0$. 
\begin{proposition}\label{degree}
Every pair $(V,J)$ equipped with a splitting $(L_1,L_2)$ is equivalent to $(\R^4,J_0)$ equipped with the splitting $S_{\alpha}$ for some unique $\alpha \in \R^+_0$. 
\end{proposition}
\begin{proof}
Let $L_1=\{\omega\}$ and $L_2=\{\omega^{\prime}\}$ for some $2$-forms $\omega,\omega^{\prime} \in \Lambda^2 V^*$. Since the wedge product restricts to be positive definite on $L_1\oplus L_2$ we have $\omega\wedge\omega >0$ and there exists a real number $\alpha$, such that 
$
\omega^{\prime}=\alpha\omega+\phi
$
for some $2$-form $\phi$ satisfying $\omega\wedge \phi=0$ and $\phi \wedge \phi>0$. After possibly rescaling $\omega^{\prime}$ we can assume that $\phi\wedge\phi=\omega\wedge\omega$ and that $\alpha$ is nonnegative. It follows with Lemma \ref{alg} that there exists a linear map $A : V \to \R^4$ which identifies $\omega$ with $\omega_0$ and $\phi$ with $\phi_0$, in particular $A$ is complex linear. To prove uniqueness of $\alpha$ suppose $A : \R^4 \to \R^4$ satisfies $A^*\omega_0=x\omega_0$ and $A^*(\alpha \omega_0+\phi_0)=y\left(\beta\omega_0+\phi_0\right)$ for some real numbers $x,y\neq 0$ and some nonnegative real numbers $\alpha,\beta$. Then $A^*\left(\omega_0\wedge\omega_0\right)=x^2\omega_0\wedge\omega_0$ and consequently 
$$
A^*\left(\omega_0\wedge(\alpha\omega_0+\phi_0)\right)=\alpha x^2\omega_0\wedge\omega_0=xy\beta\omega_0\wedge\omega_0,
$$
which is equivalent to $\alpha x = \beta y$. We also have
$$
A^*\left((\alpha \omega_0+\phi_0)\wedge (\alpha \omega_0+\phi_0)\right)=x^2(\alpha^2+1)\omega_0\wedge\omega_0=y^2(\beta^2+1)\omega_0\wedge\omega_0,
$$
which implies $x^2=y^2$ and thus $\alpha^2=\beta^2$. Since $\alpha,\beta \geq 0$, the claim follows.   
\end{proof}
For a splitting $(L_1,L_2)$, the unique nonnegative real number $\alpha$ provided by Proposition \ref{degree} will be called the \textit{degree} of the splitting. A splitting of degree $0$ will be called \textit{orthogonal}. 
\section{Local embeddability of real analytic path geometries}
\subsection{Splittings of almost complex structures}
Let $X$ be a smooth $4$-manifold and $\AJ$ be an almost complex structure with associated rank $2$ vector bundle $\B \subset \Lambda^2TX^*$ whose fibre at $p \in X$ is the linear subspace $\Lambda_{\AJ_p}\subset \Lambda^2T_pX^*$ associated to $\mathfrak{J}_p : T_pX \to T_pX$. A \textit{splitting} of $\AJ$ consists of a pair of smooth line bundles $L_1,L_2 \subset \Lambda^2TX^*$ so that $\Lambda_\AJ=L_1\oplus L_2$. 
\subsection{Induced structure on hypersurfaces} 
A \textit{CR-structure} on a $3$-manifold $M$ consists of a rank $2$ subbundle $D\subset TM$ and a vector bundle endomorphism $I : D \to D$ which satisfies $I^2=-\mathrm{Id}_D$. A CR-structure $(D,I)$ is called \textit{nondegenerate} if $D$ is nowhere integrable, i.e.~a contact plane field. A closely related notion is that of a path geometry (see for instance~\cite{MR2003610} for a motivation of the following definition). A path geometry on a $3$-manifold $M$ consists of a pair of line subbundles $(P_1,P_2)$ of  $TM$ which span a contact plane field. A CR-structure $(D,I)$ and a path geometry $(P_1,P_2)$ on $M$ will be called \textit{compatible} if $D=P_1\oplus P_2$ and $I(P_1)=P_2$. 

Let $(L_1,L_2)$ be a splitting of the almost complex structure $\AJ$ on $X$ and $(\omega,\phi)$ a pair of $2$-forms defined on some open subset $\tilde{U} \subset X$ which span $(L_1,L_2)$. Then the pair $(\omega,\phi)$ is elliptic, i.e.~$(\omega_p,\phi_p)$ is elliptic for every point $p\in \tilde{U}$. Suppose $M \subset X$ is a hypersurface. Then Lemma \ref{huereseich} implies that the $2$-forms $(\omega,\phi)$ remain linearly independent when pulled back to $M \cap \tilde{U}$. This is useful because of the following:
\begin{lemma}\label{hans}
Let $\beta_1,\beta_2$ be smooth linearly independent $2$-forms on a $3$-manifold $M$. Then there exists a local coframing $\eta=(\eta^1,\eta^2,\eta^3)^t$ of $M$ such that $\beta_1=\eta_2\wedge\eta_1$ and $\beta_2=\eta_2\wedge\eta_3$.  
\end{lemma}
Recall that a (local) \textit{coframing on} $M$ consists of three smooth linearly independent $1$-forms defined on (some proper open subset of) $M$. 
\begin{proof}[Proof of Lemma \ref{hans}]
Let $x : U \to \mathbb{E}^3$ be local coordinates on $M$ with respect to which $\beta_1\vert_U=\ba_{1}\cdot \star \d x$ and $\beta_2\vert_U=\ba_2\cdot\star \d x$ for some smooth $\ba_i : U \to \R^3$ where $\star$ denotes the Hodge-star of Euclidean space $\mathbb{E}^3$. 
Define $e=\left(\ba_1\times\ba_2\right)/\vert\ba_1\times\ba_2\vert :  U \to \R^3$ and 
$$
\eta_1=\left(b_1\times e\right)\cdot \d x,\quad \eta_2=e\cdot \d x, \quad
\eta_3=\left(b_2\times e\right)\cdot \d x, 
$$ then $(\eta^1,\eta^2,\eta^3)$ have the desired properties. 
\end{proof}
A local coframing of $M$ obtained via Lemma \ref{hans} and some (local) choice of $2$-forms $(\omega,\phi)$ spanning $(L_1,L_2)$ will be called \textit{adapted} to the structure induced by the splitting $(L_1,L_2)$. Independent of the particular adapted local coframings are the line subbundles $P_1$ and $P_2$ of $TM$, locally defined by
$$
P_1=\left\{\eta_1,\eta_2\right\}^{\perp}, \quad P_2=\left\{\eta_2,\eta_3\right\}^{\perp}. 
$$
Call a hypersurface $M\subset X$ \textit{nondegenerate} if $D=P_1\oplus P_2$ is a contact plane field. Summarising, we have shown: 
\begin{proposition}
A nondegenerate hypersurface $M\subset X$ inherits a path geometry from the splitting $(L_1,L_2)$.  
\end{proposition}

\begin{remark}
Fixing a $(2,\! 0)$-form on $X$ allows to define a coframing on a hypersurface $M\subset X$. For the construction of the coframing and its properties see~\cite{bryantunimodular}. 
\end{remark}

\subsection{Local embeddability}

We conclude by using the Cartan-K\"ahler theorem to show that locally every real analytic path geometry is induced by an embedding into $\C^2$ equipped with the splitting $(\left\{\omega_0\right\},\left\{\phi_0\right\})$. Here $\omega_0=\Re(\d z^1\wedge \d z^2)$ and $\phi_0=\Im(\d z^1 \wedge \d z^2)$ where $z=(z^1,z^2)$ are standard coordinates on $\C^2$. Writing $z^1=x^1+\i x^2$ and $z^2=x^3+\i x^4$ for standard coordinates $x=(x^i)$ on $\R^4$, we have
$$
\omega_0=\d x^1\wedge \d x^3 -\d x^2\wedge \d x^4, \quad \phi_0=\d x^1\wedge \d x^4+\d x^2\wedge \d x^3.
$$
In~\cite{MR1504846}, as an application of his method of equivalence, Cartan has shown how to associate a Cartan geometry to every path geometry.
\begin{defi}
Let $G$ be a Lie group and $H\subset G$ a Lie subgroup with Lie algebras $\mathfrak{h}\subset \mathfrak{g}$. A \textit{Cartan geometry of type} $(G,\! H)$ on a manifold $M$ consists of a right principal $H$-bundle $\pi : B \to M$ together with a $1$-form $\theta\in\mathcal{A}^1(B,\mathfrak{g})$ which satisfies the following conditions:
\begin{enumerate}
\item[(i)] $\theta_b : T_b B \to \mathfrak{g}$ is an isomorphism for every $b \in B$,
\item[(ii)] $\theta(\bX_v)=v$ for every fundamental vector field $\bX_v$, $v \in \mathfrak{h}$, 
\item[(iii)] $(R_h)^{*}\theta=\Ad_\mathfrak{g}(h^{-1})\circ\theta$. 
\end{enumerate}
Here $\Ad_\mathfrak{g}$ denotes the adjoint representation of $G$. The $1$-form $\theta$ is called the \textit{Cartan connection} of the Cartan geometry $(\pi : B\to M,\theta)$. 
\end{defi}
Denote by $\O \subset \rm{SL}(3,\R)$ the Lie subgroup of upper triangular matrices. In modern language Cartan's result is as follows (for a proof see~\cite{MR1358612,MR2003610}):
\begin{theorem}[Cartan]
Given a path geometry $(M,P_1,P_2)$, then there exists a Cartan geometry $(\pi : B \to M,\theta)$ of type $(\rm{SL}(3,\R),\! \O)$  which has the following properties: 
Writing
$$
\theta=\left(\begin{array}{ccc}\theta^0_0 & \theta^0_1 & \theta^0_2 \\ \theta^1_0 & \theta^1_1 & \theta^1_2 \\ \theta^2_0 & \theta^2_1 & \theta^2_2\end{array}\right),
$$
\begin{itemize}
\item[(i)] for any section $\sigma : M \to  B$, the $1$-form $\phi=\sigma^*\theta$ satisfies $P_1=\left\{\phi^2_1,\phi^2_0\right\}^{\perp}$ and $P_2=\left\{\phi^1_0,\phi^2_0\right\}^{\perp}$. Moreover $\phi^1_0\wedge\phi^2_0\wedge\phi^2_1$ is a volume form on $M$.  
\item[(ii)]
The curvature $2$-form $\Theta=\d\theta+\theta\wedge\theta$ satisfies
\begin{equation}\label{strucpath0}
\Theta=\left(\begin{array}{ccc} 0 & \mathcal{W}_1\, \theta^1_0\wedge\theta^2_0 & (\mathcal{W}_2\theta^1_0+\mathcal{F}_2\theta^2_1)\wedge\theta^2_0 \\ 0 & 0 & \mathcal{F}_1\, \theta^2_1 \wedge \theta^2_0 \\  0 & 0 &0\end{array}\right)
\end{equation}
for some smooth functions $\mathcal{W}_1, \mathcal{W}_2, \mathcal{F}_1, \mathcal{F}_2 : B \to \R$. 
\end{itemize} 
\end{theorem}

Using this result and the Cartan-K\"ahler theorem we obtain local embeddability in the real analytic category: 
\begin{theorem}\label{pgemb} 
Let $(\M,P_1,P_2)$ be a real analytic path geometry. Then for every point $p \in\M$ there exists a $p$-neighbourhood $U_p \subset M$ and a real analytic embedding $\varphi : U_p \to \mathbb{C}^2$ such that the path geometry induced by the splitting $(\left\{\omega_0\right\},\left\{\phi_0\right\})$ is $(P_1,P_2)$ on $U_p$. 
\end{theorem} 

\begin{proof}
Let $(\pi : B \to \M,\theta)$ denote the Cartan geometry of the path geometry $(\M,P_1,P_2)$. On $N=B\times\mathbb{R}^4$ consider the exterior differential system with independence condition $(\mathcal{I},\zeta)$ where $\zeta=\zeta^1\wedge\zeta^2\wedge\zeta^3$ with $\zeta^1=\theta^1_0,\zeta^2=\theta^2_0,\zeta^3=\theta^2_1$ and the differential ideal $\mathcal{I}$ is generated by the two $2$-forms $$
\chi_1=\theta^2_0\wedge\theta^1_0-\omega_0, \quad \chi_2=\theta^2_0\wedge\theta^2_1-\phi_0.
$$
The dual vector fields to the coframing ($\theta^i_k,\d x^l)$ of $N$ will be denoted by $(\bT^i_k,\partial_{x^l})$. Let $G_k(TN) \to \N$ be the Grassmann bundle of $k$-planes on $N$ and 
$
G_3(TN,\zeta)=\left\{E \in G_3(TN) \,\vert \zeta_E \neq 0 \right\}
$
where $\zeta_E$ denotes the restriction of $\zeta$ to the $3$-plane $E$. Let $V^k(\mathcal{I})$ denote the set of $k$-dimensional \textit{integral elements} of $\mathcal{I}$, i.e.~those $E \in G_k(TN)$ for which $\beta_E=0$ for every form $\beta \in \mathcal{I}^k=\mathcal{I}\cap \mathcal{A}^k(N)$. The flag of integral elements $F=\left(E^0,E^1,E^2,E^3\right)$ of $\mathcal{I}$ given by 
$E^{0}=\{0\}, \;E^1=\{\bv_1\},\;E^2=\{\bv_1,\bv_2\},\;E^3=\{\bv_1,\bv_2,\bv_3\}$ where
$$
\aligned
\bv_1&=\bT^1_0+\bT^2_0+\bT^2_1+\partial_{x^4},\\
\bv_2&=\bT^0_0+\bT^1_0-\bT^2_1+\partial_{x^1}+\partial_{x^2},\\
\bv_3&=\bT^1_1-\bT^2_1+\partial_{x^1},
\endaligned
$$
has Cartan characters
$
(s_0,s_1,s_2,s_3)=(0,2,4,3)$.
Therefore, by Cartan's test, $V^3(\mathcal{I})$ has codimension at least $8$ at $E^3$. However the forms of $\mathcal{I}^3$ which impose independent conditions on the elements of $G_3(TN,\zeta)$ are the eight $3$-forms
$
\d\chi_i, \chi_i\wedge\zeta^k, i=1,2$, $k=1,2,3. 
$
It follows that $V^3(\mathcal{I}) \cap G_3(TN,\zeta)$ has codimension $8$ in $G_3(TN)$. Moreover computations show that $V^3(\mathcal{I})\cap G_3(TN,\zeta)$ is a smooth submanifold near $E^3$, thus the flag $F$ is K\"ahler regular and therefore the ideal $\mathcal{I}$ is involutive. Pick points $p \in\M$ and $q=(b,0) \in N$ with $\pi(b)=p$. By the Cartan-K\"ahler theorem there exists a $3$-dimensional integral manifold $\bar{\psi}=(\bar{s},\bar{\varphi}):\Sigma\to B\times\mathbb{R}^4$ of $(\mathcal{I},\zeta)$ passing through $q$ and having tangent space $E^3$ at $q$. Every volume form on $\M$ pulls back under $\pi$ to a nowhere vanishing multiple of $\zeta$. Since $\bar{\phi}^*\zeta=\bar{s}^*\zeta\neq 0$, $\pi \circ \bar{s}: \Sigma \to \M$ is a local diffeomorphism. Therefore $p \in\M$ has a neighbourhood $U_p$ on which there exists a real analytic immersion $\psi=(s,\varphi) : U_p \to B\times \mathbb{R}^4$ such that the pair $(\psi,U_p)$ is an integral manifold of the EDS $(N,\mathcal{I},\zeta)$ and $s$ a local section of $\pi : B \to M$. After possibly shrinking $U_p$ we can assume that $\varphi$ is an embedding. Since by construction 
$
\varphi^*(\omega_0+\i\phi_0)=s^*(\theta^2_0\wedge(\theta^1_0+i\theta^2_1)),
$
it follows that the path geometry induced by $\varphi$ is $(P_1,P_2)$ on $U_p$. 
\end{proof}
\begin{remark} Every nondegenerate hypersurface $M \subset \C^2$ also inherits a CR-structure $(D,I)$ from the complex structure $J$ on $\C^2$: For every $p \in M$ define $D_p$ to be the largest $J_p$-invariant subspace of $T_pM$ and $I_p$ to be the restriction of $J_p$ to $D_p$. Then $(D,I)$ is easily seen to be compatible with the path geometry induced on $M$ by $(\left\{\omega_0\right\},\left\{\phi_0\right\})$. 
\end{remark}
Using this remark and the Theorem \ref{pgemb} we get the well-known: 
\begin{corollary}\label{cremb}
Let $(D,I)$ be a nondegenerate real analytic CR-structure on a $3$-manifold $M$. Then for every point $p \in M$ there exists a $p$-neighbourhood $U_p$ and a real analytic embedding $\varphi : U_p \to \mathbb{C}^2$, such that $(D,I)$ is the CR-structure on $U_p$ induced by the embedding $\varphi$.   
\end{corollary}
\begin{proof}
Pick a line bundle $P_2 \subset D$, define $P_1=I(P_2)$ and apply Theorem \ref{pgemb}. 
\end{proof}
\begin{remark}
Corollary~\ref{cremb} also holds without the nondegeneracy assumption and in higher dimensions~\cite{MR0460724}. 
In~\cite{0305.35017}, Nirenberg has constructed a smooth nondegenerate $3$-dimen\-sional CR-structure which is not induced by an embedding into $\mathbb{C}^2$. It follows that the real analyticity assumption in Theorem \ref{pgemb} is necessary. 
\end{remark}

\providecommand{\bysame}{\leavevmode\hbox to3em{\hrulefill}\thinspace}
\providecommand{\MR}{\relax\ifhmode\unskip\space\fi MR }
\providecommand{\MRhref}[2]{%
  \href{http://www.ams.org/mathscinet-getitem?mr=#1}{#2}
}
\providecommand{\href}[2]{#2}

\end{document}